\documentclass[11pt,letterpaper]{amsart} 
\usepackage[left=3.18cm,right=3.18cm,top=3cm,bottom=3.1cm]{geometry}
\usepackage{amsfonts}
\usepackage{amsmath}
\usepackage{amssymb}
\usepackage{mathtools}
\usepackage{amsthm}
\usepackage{pgfplots}
\pgfplotsset{compat=newest}
\usepackage[shortlabels]{enumitem}
\usepackage{tikz}   
\usepackage{tikz-cd}  
\usepackage{changepage}
\usetikzlibrary{arrows}
\usepackage{bm}
\usepackage{graphicx}
\usepackage{cases}
\usepackage{appendix}
\usepackage[colorlinks,linkcolor=black,anchorcolor=black,citecolor=black]{hyperref}
\usepackage[numbers]{natbib}
\usepackage{chapterbib}
\usepackage{fancyhdr}
\usepackage{mathrsfs}
\numberwithin{equation}{section}

\theoremstyle{definition}
\newtheorem{thm}{Theorem}[section]

\newtheorem{defn}[thm]{Definition}
\newtheorem{lem}[thm]{Lemma}
\newtheorem{prop}[thm]{Proposition}
\newtheorem{cor}[thm]{Corollary}
\newtheorem*{rmk}{Remark}

\newcommand{\R}{\mathbb{R}}  
\newcommand{\N}{\mathbb{N}}  

\newcommand{\p}{\partial}  

\newcommand{\dif}{\textup{d}} 

\newcommand{\Hau}{\mathcal{H}} 
\newcommand{\diam}{\textup{diam}} 
\newcommand{\dist}{\textup{dist}}

\newcommand{\sing}{\textup{sing }}
\newcommand{\reg}{\textup{reg }}
\newcommand{\spt}{\textup{spt }}

\usepackage{xcolor}
\usepackage{listings}

\begin{document}
\title[Minkowski content estimates]{Minkowski content estimates for generic area minimizing hypersurfaces}
\author{Xuanyu Li}
\address{Department of Mathematics, Cornell University, Ithaca, NY 14853, USA}
\email{xl896@cornell.edu}
\maketitle

\begin{abstract}
Let $\Gamma$ be a smooth, closed, oriented, $(n-1)$-dimensional submanifold of $\R^{n+1}$. It was shown by Chodosh-Mantoulidis-Schulze \cite{CMSb} that one can perturb $\Gamma$ to a nearby $\Gamma'$ such that all minimizing currents with boundary $\Gamma'$ are smooth away from a set with Hausdorff dimension less than $n-9$. We prove that the perturbation can be made such that the singular set of the minimizing current with boundary $\Gamma'$ has Minkowski dimension less than $n-9$.
\end{abstract}

\section{Introduction}
Let $\Gamma^{n-1}\subset\R^{n+1}$ be a smooth, closed, oriented submanifold. It is well known that there exists an integral current $T$ with compact support which has least area among integral currents with boundary $\Gamma$; (see \cite{Feda,Sim}). Such currents are called minimizing currents. They are known to be smooth away from a set of codimension 7; see \cite{Fedb}. 
However, there exist examples by Bombieri-De Giorgi-Giusti \cite{BDGG} which show that such minimizers can fail to be smooth.\par

The fine structure of singular sets of minimizing currents have been extensively investigated. The blow up of a minimizing current at singular points is no longer a tangent plane, but a tangent cone. The subspace of a tangent cone on which it is translation invariant is called its spine. The singular set of a minimizing current can henceforth be stratified into subsets $\mathcal{S}^l$ on which the tangent cones have no more than $l$-dimensional spines. A fundamental result of Almgren shows that the Hausdorff dimension of each $\mathcal{S}^l$ is no more than $l$ \cite{Alm}; see also \cite{Whi}. It was further shown by Cheeger-Naber \cite{CN} and Naber-Valtorta \cite{NV} that the same dimensional restriction holds for Minkowski dimensions and $\mathcal{S}^l$ is $l$-countable rectifiable.\par

On the other hand, it was shown that the singular sets of minimizing currents can be perturbed away. This was first proved by Hardt-Simon that after perturbing the boundary slightly in $\R^8$, the resulting minimizing currents are smooth \cite{HSb}; see also the work of Smale \cite{Sma} on Riemannian manifolds. Recently, it has been shown by Chodosh-Mantoulidis-Schulze \cite{CMSa} that the same result holds for minimizing currents in ambient dimensions 8,9,10. Later, the same authors also proved that the perturbation can decrease Hausdorff dimension of singular stratum by a dimensional constant greater than 2 \cite{CMSb}. The authors call such a result \emph{generic regularity} for the Plateau problem.
We also refer readers to \cite{CLS,LWa,LWb} for generic regularity for min-max minimal hypersurfaces. \par

In this paper, we prove an improvement of the results in \cite{CMSb}. That is, the above generic regularity not only holds for Hausdorff dimensions, but also for Minkowski dimensions. Specifically, we have
\begin{thm}\label{1}
    Let $\Gamma^{n-1}\subset\R^{n+1}$ be a smooth, closed, oriented submanifold. There exist arbitrary small perturbations $\Gamma'$ of $\Gamma$ (as $C^{\infty}$ graphs in the normal bundle of $\Gamma$) such that every minimizing integral $n$-current with boundary $[\![\Gamma']\!]$ is of the form $[\![M']\!]$ for a smooth, precompact, oriented hypersurface $M'$ with $\p M'=\Gamma'$, and 
    \[ \sing M'=\emptyset, \textup{ if } n+1\leqslant10, \textup{ else }\dim_{\textup{Min}}\sing M'\leqslant n-9-\epsilon_n.\]
    In fact the singular strata $\mathcal{S}^l(M')$ of $M'$ satisfy 
    \[ \mathcal{S}^0(M')=\mathcal{S}^1(M')=\mathcal{S}^2(M')=\emptyset,\quad \dim_{\textup{Min}}\mathcal{S}^l_{\epsilon}(M')\leqslant l-2-\epsilon_n\textup{ for }l\geqslant3\textup{ and }\epsilon>0. \]
    Here $\dim_{\textup{Min}}$ denotes the Minkowski dimension, and $\mathcal{S}^l_{\epsilon}(M')$ denotes the quantitative singular stratum; see Section 2 for detailed definitions. Note that for $\epsilon$ small, $\mathcal{S}^{n-7}_{\epsilon}(M')=\sing M'$ by the quantitative $\epsilon$-regularity theorem \cite[Theorem 6.2]{CN}.
\end{thm}

There, the dimensional constants $\epsilon_n$ are given explicitly by 
\[\kappa_n=\frac{n-2}{2}-\sqrt{\frac{(n-2)^2}{4}-(n-1)},\quad \epsilon_n=\kappa_n-1\in(0,1].\]
These constants were obtained by Z. Wang \cite{Wan} in their study of Jacobi fields on area minimizing hypercones, see also \cite{Simb}. 

In \cite{CMSb}, the result follows from the combination of two independent results about a family of minimizers. The first one is a Hausdorff dimension estimate on a union of singular strata of minimizers. The second one is a super Holder continuity of a timestamp function coming from perturbing the boundary of an area minimizer. Since the second one is irrelevant to which measure to use, we only need to improve the first one to a Minkowski content version:
\begin{thm}\label{2}
    Let $\mathscr{F}$ be a family of minimizing $n$-currents in $B_R(0)\subset\R^{n+1}$ whose supports are pairwise disjoint. Assume there exist $r,\Lambda>0$, such that for each $T\in\mathscr{F}$ and $x\in\sing T$
    \begin{enumerate}[(a)]
        \item $T\llcorner B_r(x)$ is a minimizing boundary;
        \item $\Theta_T(x,1)<\Lambda$,
    \end{enumerate}    
    Then given $l\in\lbrace0,1,\dots,n-7\rbrace,\epsilon,\epsilon'\in (0,1)$, we have 
    \[ \Hau^{n+1}\big(U_r(\mathcal{S}^l_{\epsilon,r}(\mathscr{F}))\big)\leqslant C(n,l,\epsilon,\epsilon',r,R,\Lambda)r^{n+1-l-\epsilon'},\]
    where $U_r(A)$ stands for the $r$-neighborhood of a set $A\subset\R^{n+1}$. In particular, we get 
    \[\mathcal{M}^{*l+\epsilon'}(\mathcal{S}^l_{\epsilon}(\mathscr{F}))<\infty,\]
    There $\Theta_T(x,1)$ is the density ratio of $T$ at $x$ and $\mathcal{M}^{*s}$ denotes $s$-dimensional upper Minkowski content; see Section 2 for definitions.
\end{thm}

In order to prove Minkowski content estimates like this, a typical way is to construct some covering of balls with same radius; see e.g. \cite{CN,NV}. The key tool (Lemma \ref{L:cluster1}) used to construct this covering is a quantitative version of an infinitesimal density drop result obtained in \cite[Proposition 3.3]{CMSa}; see Proposition \ref{density drop}. Our observation is that, Proposition \ref{density drop}, when applied to two minimizing cones with different centers, implies that the two minimizing cones must be identical.

This observation, together with a straightforward contradiction argument, implies that, if a minimizing current $T$ is sufficiently cone-like near a point, any nearby point around which another current $T'$ (that does not cross $T$ smoothly) that is also cone-line have to lie in a smaller neighborhood of a low dimensional subspace.
Therefore, we can refine the covering of singular strata using a uniformly bounded number of balls of same radius. 
On the other hand, when we are looking into smaller scales of a current near a singular point, those scales at which the current are not cone-like only appear finite times.
Using these, we are able to construct the desired covering iteratively.

In \cite{CMSb}, combining the two facts mentioned above together with a key measure theoretic lemma regarding estimates of the Hausdorff dimension of the image of a super Holder function by Figalli-Ros Oton-Serra \cite[Proposition 7.7]{FROS} gives the desired Hausdorff dimension estimates of generic area minimizing currents. For our purpose, we prove a Minkowski content version of the key lemma as follows.
In the following, a modulus of continuity refers to an increasing function $\omega:\R_+\rightarrow\R_+$ with $\lim_{r\rightarrow0}\omega(r)=0$.
\begin{prop}\label{3}
    Let $A\subset\R^n$ and $f:A\rightarrow[-1,1]$. Assume that for some $l\in(0,n]$ and $\alpha>0$ we have
    \begin{enumerate}[(a)]
        \item $\mathcal{M}^{*l}(A)<\infty$;
        \item There exists a modulus of continuity $\omega$ such that $$\vert f(x)-f(x_0)\vert\leqslant\omega(\vert x-x_0\vert)\vert x-x_0\vert^{\alpha}\textup{ for all }x,x_0\in A.$$
    \end{enumerate}
    Then:
    \begin{enumerate}[(1)]
        \item If $l\leqslant\alpha$, we have that $\mathcal{M}^{l/\alpha}(f(A))=0$;
        \item If $l>\alpha$, we have that for $\Hau^1$ a.e. $t\in[-1,1]$, we have $$\mathcal{M}^{l-\alpha+\delta}(A\cap f^{-1}(t))=0,\textup{ for all }\delta>0.$$
    \end{enumerate}
\end{prop}

\subsection{Outline of the paper}
This paper is organized as follows. In Section \ref{S:preliminaries}, we introduce basic definitions of measures and integral currents used in this paper. We also introduce the refined structure of singular sets of minimizing currents. In Section \ref{SS:Proof of Thm2}, we prove Theorem \ref{2}, the Minkowski content estimate of singular strata of a family of minimizing currents. In Section \ref{SS:proof of Prop3}, we prove Proposition \ref{3}, the Minkowski content estimates of the image of a super Holder function. Finally, we prove our main theorem in Section \ref{SS: proof of Thm1}. We include some previous results used in this paper in the Appendix.

\subsection*{Acknowledgement}
The author would like to thank 
Professor Xin Zhou for suggesting this problem 
and many enlightening discussions. 
The author also wants to thank Otis Chodosh for sharing a note related to \cite{CMSa} during his early learning stage.

\section{Preliminaries}
\label{S:preliminaries}
In this section we introduce basic definitions used in this paper.

\subsection{Hausdorff measures and Minkowski contents}
We first briefly review the notion of Hausdorff measures and Minkowski contents. Denote $$\omega_s=\frac{\pi^{s/2}}{\Gamma(s/2+1)},$$ 
where $\Gamma(t)=\int_0^{\infty}x^{t-1}e^{-x}\dif x$ is the Gamma function. When $s$ is an integer, $\omega_s$ is precisely the volume of $s$-dimensional unit ball.
\begin{defn}
    Let $s>0$ and $A\subset\R^n$ be an arbitrary subset.
    \begin{enumerate}[(1)]
        \item For $\delta>0$, define $$\Hau^s_{\delta}(A)=\inf\left\lbrace \sum_{j=1}^{\infty}\omega_s (\diam C_j/2)^s:S\subset\cup_{i=j}^{\infty}C_j\textup{ with }\diam C_j\leqslant\delta\right\rbrace.$$
        \item The $s$-dimensional \emph{Hausdorff measure} is defined as $$\Hau^s(A)=\lim_{\delta\rightarrow0}\Hau^s_{\delta}(A).$$
    \end{enumerate}
\end{defn}

For basic properties of Hausdorff measure, one may refer to \cite{Sim}. In particular, $\Hau^n$ agrees with $n$-dimensional Lebesgue measure of $\R^n$. Hence it stands for the volume of a subset in $\R^n$. Let us turn to Minkowski content. For $A\subset\R^n$, let $$U_r(A)=\lbrace x\in\R^n:\dist(x,A)<r\rbrace$$ be the $r$-neighborhood of $A$.
\begin{defn}
    Let $s>0$ and $A\subset\R^n$ be an arbitrary subset. 
    \begin{enumerate}
        \item The $s$-dimensional \emph{upper} and \emph{lower Minkowski content} of $A$ are respectively defined as $$\mathcal{M}^{*s}(A)=\limsup_{r\rightarrow0}\frac{U_r(A)}{\omega_n r^{n-s}}, \textup{\, and\, }\mathcal{M}^s_*(A)=\liminf_{r\rightarrow0}\frac{U_r(A)}{\omega_n r^{n-s}}.$$
        \item If $$\mathcal{M}^{*s}(A)=\mathcal{M}_*^s(A),$$
        we say $A$ is s-dimensional \emph{Minkowski measurable} and denote their same value by $\mathcal{M}^s(A)$.
    \end{enumerate}
\end{defn}
Unfortunately, it turns out that neither $\mathcal{M}^{*s}$ nor $\mathcal{M}_*^s$ is a measure except for $s=0$. In fact, we have the following fact $$\mathcal{M}^{*s}(\Bar{A})=\mathcal{M}^{*s}(A),\quad \mathcal{M}_*^s(\Bar{A})=\mathcal{M}^s_*(A),$$
where $\Bar{A}$ is the closure of $A$. Hence though a single point is $s$-dimensional Minkowski zero content for $s>0$, but any countable dense subset of $\R^n$ has infinite $n$-dimensional Minkowski content.\par 
\begin{defn}
    Let $A\subset\R^n$, its \emph{upper} and \emph{lower Minkowski dimension} are respectively defined as $$\overline{\dim}_{\textup{Min}}(A)=\inf\lbrace s:\mathcal{M}^{*s}(A)=0\rbrace,\quad \underline{\dim}_{\textup{Min}}(A)=\sup\lbrace s:\mathcal{M}_*^s(A)=\infty\rbrace.$$
\end{defn}
In this paper, we only focus on upper Minkowski dimension and denote it by $\dim_{\textup{Min}}$. Note that the upper Minkowski content has the nice description:$$\mathcal{M}^{*s}(A)<\infty\iff \exists C>0, \Hau^n(A)\leqslant Cr^{n-s}\textup{ for }r \textup{ sufficiently small}.$$
\subsection{Minimizing currents}
Let $U$ be an open set in $\R^{n+1}$. We denote by $\mathbb{I}_k(U)$ the space of integral $k$-currents in $U$. This a generalization of $k$-dimensional submanifolds in $U$ first studied by Federer and Fleming (\cite{FF}). It turns out that $\mathbb{I}_n(U)$ is a good space for studying variational problems such as Plateau problem because it has good compactness property. We refer the reader to standard references \cite{Feda} and \cite{Sim}. For $T\in \mathbb{I}_k(U)$, we denote by $\Vert T\Vert$ the associated Radon measure. The weak convergence of integral currents is denoted by $\rightharpoonup$.\par
\begin{defn}
    If the density function of $T\in\mathbb{I}_k(U)$ equals to 1 $\Hau^k$ a.e. on $\spt T$, we say $T$ is \emph{multiplicity one}. The space of multiplicity one $k$-currents is denoted by $\mathbb{I}_k^1(U)\subset\mathbb{I}_k(U)$.
\end{defn}
We only consider the codimension one case $k=n$. 
\begin{defn}
    We say $C\in\mathbb{I}_n^1(\R^{n+1})$ is a \emph{cone} centered at $x$ if $$\eta_{x,\lambda\#}C=\eta_{x,1\#}C\textup{ for all }\lambda>0.$$
\end{defn}
We will be interested in \emph{mass minimizing currents} $T$ in $\mathbb{I}_n(U)$, or \emph{minimizing currents} for brevity, meaning that $$\Vert T\Vert(W)\leqslant\Vert T+X\Vert(W)\textup{ for all }W\Subset U\textup{ and }\spt X\subset W,\p X=0.$$
We denote by $$\Theta_T(x,r)=\frac{\Vert T\Vert(B_r(x))}{\omega_n r^n}$$ the \emph{density ratio} of $T$. By the well-known monotonicity formula, $\Theta_T(x,r)$ is increasing for $B_r(x)\cap \spt\p T=\emptyset$. Hence the limit $\Theta_T(x)=\lim_{r\rightarrow0}\Theta(x,r)$ exists anywhere on $\spt T\setminus\spt\p T$, call the \emph{density} of $T$ at $x$. Moreover, $\Theta_T(x)$ is upper semicontinuous on both arguments $x$ and $T$. Another frequently used fact is that for any sequence $r_i\rightarrow0$ and $x\in\spt T\setminus\spt\p T$, we can pass to a subsequence to derive $$C=\lim_{i\rightarrow\infty}\eta_{x,r_i\#}T\textup{ exists and is a cone centered at 0 with }\Theta_C(0)=\Theta_T(x).$$
Note also that the study of multiplicity one minimizing currents can be reduced to the study of \emph{minimizing boundaries}, see \cite[\S 37]{Sim}.

\subsection{Singular sets and its stratification}
We first define the regular part and singular part of a minimizing current.
\begin{defn}
    Let $U\subset\R^{n+1}$ be an open set and $T\in\mathbb{I}_n(U)$ be minimizing. Denote 
    \begin{equation}
        \begin{aligned}
            \reg T=\lbrace x\in U\cap\spt T\setminus\spt\p T: \textup{ for some }m\in\N, r>0\textup{ and }n\textup{-submanifold }\Sigma,&\\\textup{we have }T\llcorner B_r(x)=m[\![\Sigma]\!]\rbrace,&
        \end{aligned}\nonumber
    \end{equation}
    $$\text{and}\quad \sing T=U\cap\spt T\setminus(\reg T\cup\spt\p T).$$
\end{defn}
The central objects studied in this paper are refined subsets of $\sing T$ which rely on the decomposition of tangent cones of $T$.
\begin{defn}
    Let $C\in\mathbb{I}_n^1(\R^{n+1})$ be a minimizing cone. We call its top density part \emph{spine}:$$\textup{spine }C=\lbrace x\in\spt C:\Theta_C(x)=\Theta_C(0)\rbrace.$$
\end{defn}
One can show that $\textup{spine }C$ is a subspace of $\R^{n+1}$ on which $C$ is translation invariant. Moreover, by \cite{JSim}, $\dim\textup{spine }C\leqslant n-7$.\par 
\begin{defn}
    Let $T$ be a minimizing boundary in $U$.
    \begin{enumerate}[(1)]
        \item The $l$-th \emph{singular stratum} of $T$ is defined by $$\mathcal{S}^l(T)=\lbrace x\in\spt T:\dim\textup{spine }C\leqslant l \textup{ for all tangent cones }C\textup{ of }T\textup{ at }x\rbrace.$$
        \item For $\epsilon,r>0$, we further define the \emph{quantitative singular stratum} by$$\mathcal{S}^l_{\epsilon,r}(T)=\lbrace x\in\spt T:\mathcal{F}_{B_1(0)}(\eta_{x,s\#}T,\mathscr{C}^{l+1})>\epsilon\textup{ for all }s\in[r,1] \textup{ and }B_1(0)\subset\eta_{x,s}(U)\rbrace,$$ $$\text{and}\quad \mathcal{S}^l_{\epsilon}(T)=\bigcap_{r>0}\mathcal{S}^l_{\epsilon,r}(T),$$
        where $\mathscr{C}^{l+1}$ is the space of minimizing cones in $\R^{n+1}$ with $\dim\textup{spine }C\geqslant l+1$ for each $C\in\mathscr{C}^{l+1}$ and $\mathcal{F}_{B_1(0)}$ is the flat metric on $B_1(0)$.
    \end{enumerate}
\end{defn}
Note that we have the following monotonicity of $\mathcal{S}_{\epsilon,r}^l(T)$:$$\mathcal{S}^l_{\epsilon',r'}(T)\subset\mathcal{S}^l_{\epsilon,r}(T),\textup{ if }\epsilon'>\epsilon,r'<r$$
We have the stratification of $\sing T$ via $\mathcal{S}^l(T)$ as follows:$$\mathcal{S}^0(T)\subset\mathcal{S}^1(T)\subset\cdots\subset\mathcal{S}^{n-7}(T)=\sing T.$$
By classical results (\cite{Feda}), $\dim_{\textup{Hau}}\mathcal{S}^l(T)\leqslant l$. The notion of quantitative singular stratum is taken from \cite{CN,NV}. There authors proved that the same dimension bound holds for Minkowski dimensions on $\mathcal{S}_{\epsilon}^l(T)$.\par
In this paper, what we really concern is a family $\mathscr{F}$ of minimizing currents. In this case, we define$$\sing\mathscr{F}=\bigcup_{T\in\mathscr{F}}\sing T,\quad \mathcal{S}^l(\mathscr{F})=\bigcup_{T\in\mathscr{F}}\mathcal{S}^l(T),\quad \textup{ and }\mathcal{S}^l_{\epsilon,r}(\mathscr{F})=\bigcup_{T\in\mathscr{F}}\mathcal{S}^l_{\epsilon,r}(T).$$
When comparing two minimizing currents, we need the following concept taken from \cite{CMSa}.
\begin{defn}
    Suppose $T,T'\in\mathbb{I}_n(U)$. We say that $T,T'$ \emph{cross smoothly} at $x\in\reg T\cap\reg T'$, if for all sufficiently small $r>0$, there are points of $\reg T'$ on both sides of $\reg T$ within $B_r(x)$ and vice versa.
\end{defn}

\section{Proof of Results}
\subsection{Proof of Theorem \ref{2}}
\label{SS:Proof of Thm2}
We first need the following observation due to Proposition \ref{density drop}.
\begin{lem}\label{rigidity}
    Suppose $C, C'\in \mathbb{I}_n^1(\R^{n+1})$ are two non-flat minimizing cones centered at $0,x$ respectively. Then$$C=\pm C',x\in\textup{spine }C\textup{ if and only if } C\textup{ does not cross }C'\textup{ smoothly}.$$
\end{lem}
We extend it into a quantitative version.

\begin{lem}\label{L:cluster1}
    Given $\epsilon,\gamma\in(0,1)$, there exists $\eta=\eta(n,l,\gamma,\epsilon)\in(0,1)$ with the following property.\\
    Suppose $T$ is a minimizing boundary in $B_{\eta^{-1}}(0)$ with $0\in\mathcal{S}^l_{\epsilon,\gamma}(T)$ and $$\Theta_T(0,1)-\Theta(0,\gamma)<\eta.$$Then there exists a $\leqslant l$-dimensional subspace $\Pi$ of $\R^{n+1}$ such that:\par
    If $T'$ is another minimizing boundary in $B_{\eta^{-1}}(0)$ that does not cross $T$ smoothly. If $x\in B_1(0)\cap\sing T'$ satisfies $$\Theta_{T'}(x,1)-\Theta_{T'}(x,\gamma)<\eta,$$then $$x\in U_{\gamma}(\Pi).$$
\end{lem}
\begin{proof}
    We prove by contradiction. Suppose the lemma does not hold, then there exists a series of minimizing boundaries $T_j$ in $B_j(0)$ with $0\in\mathcal{S}_{\epsilon,\gamma}^l(T_j),x_j\in B_1(0)\cap\sing T'$ and 
    $$\Theta_{T_j}(0,1)-\Theta_{T_j}(0,\gamma)<1/j.$$
    After passing to a subsequence, we can assume $T_j\rightharpoonup T$ a minimizing boundary in $\R^{n+1}$. By \cite[Theorem 34.5]{Sim}, the associated measure $\Vert T_j\Vert$ converges to $\Vert T\Vert$ in the weak sense of Radon measure. Hence
    $$\Vert T\Vert(B_1(0))\leqslant\liminf_{j\rightarrow\infty}\Vert T_j\Vert(B_1(0))\textup{ and }\Vert T\Vert(B_{\gamma'}(0))\geqslant\Vert T\Vert(\Bar{B}_{\gamma}(0))\geqslant\limsup_{j\rightarrow\infty}\Vert T_j\Vert(\Bar{B}_{\gamma}(0))$$
    for some $\gamma'$ slightly larger than $\gamma$. As a result, we have 
    $$0\leqslant\Theta_T(0,1)-\Theta_T(0,\gamma')\leqslant\liminf_{j\rightarrow\infty}(\Theta_{T_j}(0,1)-\Theta_{T_j}(0,\gamma))=0.$$
    Hence $T$ is a minimizing cone in $B_1(0)\setminus B_{\gamma'}(0)$. By the unique continuation of minimal hypersurfaces, $T$ itself is a minimizing cone $C$ in $\R^{n+1}$. Moreover, since $0\in\mathcal{S}_{\epsilon,\gamma}^l(T_j)$ for each $j$, we know that 
    $$\mathcal{F}_{B_1(0)}(T_j,\mathscr{C}^{l+1})>\epsilon\implies\mathcal{F}_{B_1(0)}(C,\mathscr{C}^{l+1})>\epsilon.$$
    Hence if we fix $\Pi=\textup{spine } C$, then $\dim\Pi\leqslant l$. Now, by the contradiction assumption, there exists minimizing boundaries $T_j'$ in $B_j(0)$ that does not cross $T_j$ smoothly, $x_j\in B_1(0)\cap\sing T_j$ with 
    $$\Theta_{T_j'}(x_j,1)-\Theta_{T_j'}(x_j,\gamma)<1/j;$$
    but 
    $$x_j\notin U_{\gamma}(\Pi).$$
    After passing to a further subsequence, we can assume $T_j'\rightharpoonup T'$ a minimizing boundary in $\R^{n+1}$ and $x_j\rightarrow x\in B_1(0)$. By Allard's regularity theorem \cite[Theorem 24.2]{Sim}, we have $x\in \sing T'$. Moreover, for some $r$ slightly smaller than 1 and $\gamma''$ slightly larger than $\gamma'$, we have $$\Bar{B}_{\gamma}(x_j)\subset\Bar{B}_{\gamma'}(x)\subset B_{\gamma''}(x)\subset B_r(x)\subset B_1(x_j).$$
    Therefore, 
    $$\Vert T'\Vert(B_r(x))\leqslant\liminf_{j\rightarrow\infty}\Vert T_j'\Vert(B_r(x))\leqslant\liminf_{j\rightarrow\infty}\Vert T_j'\Vert(B_1(x_j)),$$ 
    $$\Vert T'\Vert(B_{\gamma''}(x))\geqslant\Vert T'\Vert(\Bar{B}_{\gamma'}(x))\geqslant\limsup_{j\rightarrow\infty}(\Bar{B}_{\gamma'}(x))\geqslant\limsup_{j\rightarrow\infty}\Vert T'\Vert(B_{\gamma}(x_j)).$$
    Again by 
    $$0\leqslant\Theta_{T'}(x,r)-\Theta_{T'}(x,\gamma)\leqslant\liminf(\Theta_{T_j'}(x_j,1)-\Theta_{T_j'}(x_j,\gamma))\leqslant0$$
    and $x\in\sing T'$, we know that $T'$ is a non-flat minimizing cone $C'$ centered at $x$. Moreover, since $T_j$ and $T_j'$ converge locally smoothly on $\reg C,\reg C'$ respectively, we know that $C$ and $C'$ cannot cross smoothly because $T_j, T_j'$ were not. Hence by Lemma \ref{rigidity}, we have $C=C'$ and $x\in\textup{spine } C=\Pi$. This contradicts the fact that 
    $$x_j\notin U_{\gamma}(\Pi).$$
\end{proof}
What we really need is the following rescaled version. Suppose $T,T'$ are minimizing boundaries in $B_{\eta^{-1}\gamma^{j-1}}(x)$, by applying Lemma \ref{L:cluster1} to $\eta_{x,\gamma^{j-1}\#}T,\eta_{x,\gamma^{j-1}\#}T'$, we get
\begin{cor}\label{cluster2}
    Given $\epsilon,\gamma\in(0,1)$, there exists $\eta=\eta(n,l,\gamma,\epsilon)\in(0,1)$ with the following property.\par
    Suppose $T$ is a minimizing boundary in $B_{\eta^{-1}\gamma^{j-1}}(x)$, $x\in\mathcal{S}^l_{\epsilon,\gamma^j}(T)$ and $$\Theta_T(x,\gamma^{j-1})-\Theta(x,\gamma^j)<\eta.$$Then there exists a $\leqslant l$-dimensional subspace $\Pi$ of $\R^{n+1}$ such that:\par 
    If $T'$ is another minimizing boundary in $B_{\eta^{-1}\gamma^{j-1}}(x)$ that does not cross $T$ smoothly, for any $x'\in B_{\gamma^{j-1}}(x)\cap\sing T'$ satisfies $$\Theta_{T'}(x',\gamma^{j-1})-\Theta_{T'}(x',\gamma^j)<\eta,$$ 
    we have
    $$x'\in U_{\gamma^j}(\Pi+x).$$
\end{cor}
Now we can begin to construct explicit covering of a family of minimizing currents. The construction is taken from \cite{CN}.
\begin{lem}[{cf. \cite[Lemma 7.1]{CN}}]\label{decomposition}
    Let $\mathscr{F}$ be a family of minimizing currents in $B_R(0)\subset\R^{n+1}$ whose supports are pairwise disjoint. Assume there exist $r,\Lambda>0$ such that for each $T\in\mathscr{F}$ and $x\in\sing T$
    \begin{enumerate}[(a)]
        \item $T\llcorner B_r(x)$ is a minimizing boundary;
        \item $\Theta_T(x,1)<\Lambda$.
    \end{enumerate}
    Then given $l\in\lbrace0,1,\dots,n-7\rbrace,\epsilon,\gamma\in (0,1)$, there exist constants $c_1(n),c_2(n), M(n,R)$,\\$K(n,l,\gamma,\epsilon,\Lambda),Q(n,l,\gamma,\epsilon,r,\Lambda)$ with the following property.\par 
    For $j\in\N$ sufficiently large,
    \begin{enumerate}[(1)]
        \item $\sing\mathscr{F}$ can be decomposed into at most $j^K$ subsets $E_{A^j}$;
        \item Each $E_{A_j}\cap\mathcal{S}_{\epsilon,\gamma^j}^l(\mathscr{F})$ can be covered by $$M(c_1\gamma^{-n-1})^Q(c_2\gamma^{-l})^{j-Q}$$balls of radius $\gamma^j$.
    \end{enumerate}
\end{lem}
\noindent\textbf{The construction of decomposition.} Fix $\eta\in(0,1)$ according to Lemma \ref{L:cluster1}. For each $T\in\mathscr{F}$, $x\in\sing T$ and $j\in\N$, define
\begin{equation}
    a_j(x)=\begin{cases}
        1,&\Theta_T(x,\gamma^{j-1})-\Theta_T(x,\gamma^j)\geqslant\eta,\\
        0,&\Theta_T(x,\gamma^{j-1})-\Theta_T(x,\gamma^j)<\eta.
    \end{cases}\nonumber
\end{equation}
The infinite dimensional tuple associated to $x$ is defined as $$A(x)=(a_1(x),a_2(x),\dots)\in 2^{\N}.$$For each infinite dimensional $0-1$ tuple $A=(a_1,a_2,\dots)\in 2^{\N}$, let $A^j=(a_1,\dots,a_j)$ be its first $j$ entries. Now, define $$E_{A^0}=\sing \mathscr{F},E_{A^j}=\lbrace x\in\sing\mathscr{F}:A^j(x)=A^j\rbrace.$$Clearly, we have the decomposition of $\sing\mathscr{F}$:$$\sing\mathscr{F}=\bigcup_{\textup{All possible }A}E_{A^j}.$$
Finally, let us take a minimal covering of $E_{A^j}\cap\mathcal{S}^l_{\epsilon,\gamma^j}(\mathscr{F})$ by balls of radius $\gamma^j$. 

\begin{proof}[Proof of Lemma \ref{decomposition}]
    We henceforth estimate the amounts of sets involved in the above procedure.
    
    \smallskip
    $(1)$ For each $T\in\mathscr{F}$ and $x\in \sing T$, using assumption (a) we have $$\sum_{j=1}^{\infty}(\Theta_T(x,\gamma^{j-1})-\Theta_T(x,\gamma^j))=\Theta_T(x,1)-\Theta_T(x)\leqslant\Lambda.$$Hence, the $j$'s such that $\Theta_T(x,\gamma^{j-1})-\Theta_T(x,\gamma^j)\geqslant\eta$ can only occur no more than $K=[\Lambda/\eta]$ times. This means that $\vert A^j\vert\leqslant K$ provided $E_{A^j}\ne\emptyset$. For $j>K$, the amount of all such possible $A^j$ is 
    $$\binom{j}{K}\leqslant j^K.$$

    \smallskip
    $(2)$ Fix $A=(a_1,a_2,\dots)$ such that $E_{A^j}$ is nonempty. Let $J$ be the smallest integer such that $\eta^{-1}\gamma^{J-1}<r$. We inductively define the required covering for $E_{A^j}\cap\mathcal{S}_{\epsilon,\gamma^j}^l(\mathscr{F})$. First by our assumption $E_{A^j}\cap\mathcal{S}_{\epsilon,\gamma^j}^l(\mathscr{F})\subset B_R(0)$, we can cover $E_{A^j}\cap\mathcal{S}_{\epsilon,\gamma^j}^l(\mathscr{F})$ by $c_1(n)R^{n+1}$ balls of radius 1 with centers in $E_{A^j}\cap\mathcal{S}_{\epsilon,\gamma^j}^l(\mathscr{F})$. Suppose for $k<j$, we have already defined a covering of $E_{A^j}\cap\mathcal{S}_{\epsilon,\gamma^j}^l(\mathscr{F})$ by balls of radius $\gamma^{k-1}$ with centers in $E_{A^j}\cap\mathcal{S}_{\epsilon,\gamma^j}^l(\mathscr{F})$. For each ball $B_{\gamma^{k-1}}(x)$ in the covering constructed in the last step, we have exactly two cases.
    \begin{itemize}
        \item [1.] If $k\leqslant J$ or $a_k=1$, we simply cover $B_{\gamma^{k-1}}(x)\cap E_{A^j}\cap\mathcal{S}^l_{\epsilon,\gamma^j}(\mathscr{F})$ by $c_1\gamma^{-n}$ balls of radius $\gamma^k$ with centers in $E_{A^j}\cap\mathcal{S}^l_{\epsilon,\gamma^j}$.
        \item [2.] Otherwise, elements in $\mathscr{F}$ are minimizing boundaries in $B_{\eta^{-1}\gamma^{k-1}}(x)$. Also, if $T'\in\mathscr{F},x'\in\mathcal{S}^l_{\epsilon,\gamma^j}(T')\cap E_{A^j}$, we know that $$\Theta_{T'}(x',\gamma^{k-1})-\Theta_{T'}(x',\gamma^k)<\eta.$$Applying Corollary \ref{cluster2}, we get $$B_{\gamma^{k-1}}(x)\cap E_{A^j}\cap\mathcal{S}^l_{\epsilon,\gamma^j}(\mathscr{F})\subset U_{\gamma^k}(\Pi+x)$$for some $\leqslant l$ dimensional subspace $\Pi\subset\R^{n+1}$. Hence $B_{\gamma^{k-1}}(x)\cap E_{A^j}\cap\mathcal{S}^l_{\epsilon,\gamma^j}(\mathscr{F})$ can be covered by $c_2(n)\gamma^{-l}$ balls of radius $\gamma^k$ with centers in $E_{A^j}\cap\mathcal{S}^l_{\epsilon,\gamma^j}(\mathscr{F})$.
    \end{itemize}
    Since there are no more than $K$ nonzero entries in $A^j$, we see that case 1 can occur at most $Q=J+K$ times. Hence, we have constructed a covering of $E_{A^j}\cap\mathcal{S}^l_{\epsilon,\gamma^j}(\mathscr{F})$ by $$M(c_1\gamma^{-n-1})^Q(c_2\gamma^{-l})^{j-Q}$$ balls of radius $\gamma^{j}$ with centers in $E_{A^j}\cap\mathcal{S}^l_{\epsilon,\gamma^j}(\mathscr{F})$.
\end{proof}

Now we are ready to prove Theorem \ref{2}.

\begin{proof}[Proof of Theorem \ref{2}]
    Choose $\gamma=c_2^{-2/\epsilon'}$ in Lemma \ref{decomposition}. Now for $j$ large, we have $$c_2^j=\gamma^{-\epsilon'j/2},j^K\leqslant\gamma^{-\epsilon'j/2}.$$By Lemma \ref{decomposition}, $U_{\gamma^j}(\mathcal{S}^l_{\epsilon,\gamma^j}(\mathcal{\mathscr{F}}))$ can be covered by $Mj^K(c_1\gamma^{-n-1})^Q(c_2\gamma^{-l})^{j-Q}$ balls of radius $2\gamma^j$, hence 
    \begin{equation}
        \begin{aligned}
            \Hau^{n+1}(U_{\gamma^j}(\mathcal{S}^l_{\epsilon,\gamma^j}))&\leqslant \omega_{n+1}Mj^K(c_1\gamma^{-n-1})^Q(c_2\gamma^{-l})^{j-Q}\gamma^{jn}\\&\leqslant C(n,K,Q,l,\gamma)\gamma^{j(n-l-\epsilon')}.
        \end{aligned}\nonumber
    \end{equation}
    For general $r\in(0,1)$, choose $j$ such that $\gamma^{j+1}\leqslant r<\gamma^j$, then$$U_r(\mathcal{S}^l_{\epsilon,r})\subset U_{\gamma^j}(\mathcal{S}^l_{\epsilon,\gamma^j})\implies \Hau^{n+1}(U_r(\mathcal{S}^l_{\epsilon,r}))\leqslant C\gamma^{j(n-l-\epsilon')}\leqslant Cr^{n-l-\epsilon'}.$$
\end{proof}

\subsection{Proof of Proposition \ref{3}}
\label{SS:proof of Prop3}
\begin{proof}
    Given $r>0$, we have $A\subset\cup_{x\in A}B_r(x)$. By Vitalli covering lemma \cite[Theorem 3.3]{Sim}, for some $A'\subset A$, $\lbrace B_r(x)\rbrace_{x\in A'}$ are pairwise disjoint and 
    $$A\subset\cup_{x\in A'}B_{5r}(x).$$
    Relabelling, we may write $A\subset \cup_{i\in I}B_r(x_i)$ with $x_i\in A$ and $\lbrace B_{r/5}(x_i)\rbrace_{i\in I}$ pairwise disjoint. By assumption, for $r$ small, we have 
    $$\sum_{i\in I}\omega_n(r/5)^n\leqslant\Hau^n(U_{r/5}(A))\leqslant C(r/5)^{n-l}.$$
    Hence the index set $I$ must be finite with 
    $$\# I\leqslant C(n,l) r^{-l}.$$
    Moreover, by the definition of modulus of continuity, given $\epsilon>0$, there exists $r(\epsilon)>0$ small such that $\omega(r)<\epsilon$ for $0<r<r(\epsilon)$. Hence if $r<r(\epsilon)$, setting $t_i=f(x_i)$, we have $$f(A)=\bigcup_{i\in I}f(A\cap B_r(x_i))\subset\bigcup_{i\in I}(t_i-\epsilon r^{\alpha},t_i+\epsilon r^{\alpha}).$$
    
    Let us show $(1)$. We have 
    $$U_{\epsilon r^{\alpha}}(f(A))\subset\bigcup_{i\in I}(t_i-2\epsilon r^{\alpha},t_i+2\epsilon r^{\alpha}).$$
    Hence 
    $$\Hau^1(U_{\epsilon r^{\alpha}}(f(A))\leqslant4\epsilon r^{\alpha}\# I\leqslant C\epsilon r^{\alpha-l}=C\epsilon^{l/\alpha}(\epsilon r^{\alpha})^{1-l/\alpha}.$$
    This implies $\mathcal{M}^{l/\alpha}(f(A))=0$.

    For $(2)$, for $t\in[-1,1]$, set 
    $$I_{\epsilon,r}(t)=\lbrace i:t\in(t_i-\epsilon r^{\alpha},t_i+\epsilon r^{\alpha})\rbrace,$$ 
    $$N_{\epsilon,r}(t)=\sum_{i\in I}\chi_{(t_i-\epsilon r^{\alpha},t_i+\epsilon r^{\alpha})}(t)=\# I_{\epsilon,r}(t).$$
    Note that $$A\cap f^{-1}(t)=\bigcup_{i\in I}A\cap f^{-1}(t)\cap B_r(x_i)\subset \bigcup_{i\in I_{\epsilon,r}(t)}B_r(x_i).$$
    Hence $$\Hau^n(U_r(A\cap f^{-1}(t)))\leqslant C(n)r^n\# I_{\epsilon,r}(t).$$
    Let us determine for which $t$ we have that $\# I_{\epsilon,r}(t)=N_{\epsilon,r}(t)$ is a high order term of $r^{\alpha-l}$. To that end, we integrate $N_{\epsilon,r}$ over $[-1,1]$ to yield 
    $$\int_{-1}^1N_{\epsilon,r}(t)\dif t\leqslant\epsilon r^{\alpha}\# I\leqslant C\epsilon r^{\alpha-l}.$$
    Next fix $\delta>0$ small. By setting $T_{\epsilon,r}^{\delta}:=\lbrace t: N_{\epsilon,r}(t)\geqslant C\sqrt{\epsilon}r^{\alpha-l-\delta}\rbrace$, we get $\Hau^1(T_{\epsilon,r})\leqslant\sqrt{\epsilon}r^{\delta}$. We consider $\epsilon=2^{-2j}$, $r_j=r(2^{-2j})/2$ and set $T_{j,k}^{\delta}=T^{\delta}_{2^{-2j},2^{1-k}r_j}$. Hence $$\Hau^1(T^{\delta}_{j,k})\leqslant2^{-j-\delta (k-1)}r_j^{\delta}\leqslant 2^{-j-\delta (k-1)}.$$
    As a result  
    \[
    \begin{aligned}
    T^{\delta}=\bigcap_{m=1}^{\infty}\bigcup_{j=m}^{\infty}\bigcup_{k=1}^{\infty}T^{\delta}_{j,k} & \implies \\
    &\Hau^1(T^{\delta})\leqslant\sum_{j=m}^{\infty}\sum_{k=1}^{\infty}\Hau^1(T^{\delta}_{j,k})\leqslant\sum_{j=m}^{\infty}2^{-j}\sum_{k=1}^{\infty}2^{-\delta (k-1)}\rightarrow0,\textup{ as }m\rightarrow\infty.
    \end{aligned}
    \]
    If $t\in[-1,1]$ but $t\notin T^{\delta}$, then $t\notin\cup_{k=1}^{\infty}T^{\delta}_{j,k}$ for all sufficiently large $j$. Tracing back, we have 
    $$\Hau^n(U_{2^{1-k}r_j}(A\cap f^{-1}(t)))\leqslant C2^{-j}(2^{1-k}r_j)^{n+\alpha-l-\delta} \textup{ for }j\textup{ sufficiently large}.$$
    For $r>0$ sufficiently small, suppose $r<r_j$ for some $j$ large. There exists a $k\in\N$ with $2^{-k}r_j\leqslant r<2^{1-k}r_j$. Hence 
    \[
    \begin{aligned}
       \Hau^n(U_r(A\cap f^{-1}(t))) & \leqslant \Hau^n(U_{2^{1-k}r_j}(A\cap f^{-1}(t))) \\ 
       & \leqslant C2^{-j}(2^{1-k}r_j)^{n+\alpha-l-\delta}\leqslant C2^{n+\alpha-l-\delta}2^{-j}r^{n+\alpha-l-\delta}. 
    \end{aligned}
    \]
    This proves that $\mathcal{M}^{l-\alpha+\delta}(A\cap f^{-1}(t))=0$ for $t\in [-1,1]\setminus T^{\delta}$. Finally, $T=\cup_{j=1}^{\infty}T^{1/j}$ is still a $\Hau^1$ zero measure set. We have 
    \[\mathcal{M}^{l-\alpha+\delta}(A\cap f^{-1}(t))=0\textup{ for all }\delta>0\textup{ provided }t\in [-1,1]\setminus T.\]
\end{proof}
\begin{rmk}
    When $l\geqslant\alpha$, with a slight modification, we can prove
    \begin{itemize}
        \item $\mathcal{M}^{l-\alpha}_*(A\cap f^{-1}(t))=0$ for $\Hau^1$ a.e $t\in[-1,1]$;
        \item If for some $C>0$, $\omega$ satisfies 
        $$\omega(r)\leqslant\frac{C}{(\log r)^2}\textup{ for }r\textup{ small},$$
        then $\mathcal{M}^{l-\alpha}(A\cap f^{-1}(t))=0$ for $\Hau^1$ a.e $t\in[-1,1]$.
    \end{itemize}
\end{rmk}

\subsection{Proof of Theorem \ref{1}}
\label{SS: proof of Thm1}
Now we can begin to prove our main theorem. The proof goes almost identical to the proof in \cite{CMSb}.
\begin{proof}
    By \cite[Section 6]{CMSb}, we can first perturb $\Gamma$ to a nearby $\Gamma'$ such that all elements in $\mathscr{M}(\Gamma')$ is of multiplicity one. Moreover, it is possible to construct a family of perturbation $(\Gamma'_s)_{s\in[-\delta,\delta]}$ with $\Gamma'_0=\Gamma'$ and $(\Gamma_s')_{s\in[-\delta,\delta]}$ satisfies the assumption of Theorem \ref{holder continuity}. Hence we can apply Theorem \ref{2} and Proposition \ref{3} to $f=\mathfrak{t}$ on $\mathcal{S}_{1/i}^l(\mathscr{F})$ where $\mathscr{F}=\cup_{s\in[-\delta,\delta]}\mathscr{M}(\Gamma_s')$ to derive Theorem \ref{1} with $\epsilon=1/i$. Note that:
    \begin{itemize}
        \item Due to \cite{HSa}, if a multiplicity one minimizing currents $T$ has a smooth boundary $[\![\Sigma]\!]$, we have that $T$ is a smooth hypersurface in a small neighborhood of $[\![\Sigma]\!]$. Using the compactness of $\mathscr{F}$, we know that Assumption (a) of Theorem \ref{2} is satisfied.
        \item Again using the compactness of $\mathscr{F}$, let $$\Lambda=\sup_{T\in\mathscr{F},x\in\sing T}\Theta_T(x,1)$$Suppose $T_j\in\mathscr{F}$ and $x_j\in\sing T_j$ such that $\Theta_{T_j}(x_j,1)\rightarrow\Lambda$. After passing to a subsequence, we may assume $T_j\rightharpoonup
        T\in\mathscr{F}$ and $x_j\rightarrow x\in\sing T$. For some $R$ slightly larger than 1, we have $$\infty>\omega_n^{-1}\Vert T\Vert(B_R(x))\geqslant\omega_n^{-1}\limsup_{j\rightarrow\infty}\Vert T\Vert(B_1(x_j))=\Lambda$$Hence the Assumption (b) of Theorem \ref{2} is satisfied
    \end{itemize}
    As a result, we can apply Theorem \ref{2} to $\mathscr{F}$ to derive that $$\mathcal{M}^{l+1/j}(\mathcal{S}_{1/j}^l(\mathscr{F}))=0,\textup{ for all }j\in\N,l\in\lbrace0,\dots,n-7\rbrace$$By applying Proposition \ref{3} and Theorem \ref{holder continuity} to $\mathfrak{t}\big|_{\mathcal{S}_{1/j}^l(\mathscr{F})}$ and $\alpha=\kappa_n+1-1/j$, we get a $\Hau^1$ measure zero set $A_j\subset[-\delta,\delta]$ such that for all $s\in[-\delta,\delta]\setminus A_j$, we have $$\mathcal{S}^0_{1/j}(T)=\mathcal{S}^1_{1/j}(T)=\mathcal{S}^2_{1/j}(T)=\emptyset,\mathcal{M}^{l+2/j+\epsilon-2-\epsilon_n}(\mathcal{S}^l_{1/j}(T))=0,\textup{ for all }T\in\mathscr{M}(\Gamma_s),\epsilon>0$$Now $A=\cup_{j=1}^{\infty}A_j$ is still of $\Hau^1$ measure zero, and for all $s\in[-\delta,\delta]\setminus A$, we have $$\mathcal{S}^0(T)=\mathcal{S}^1(T)=\mathcal{S}^2(T)=\emptyset,\mathcal{M}^{l+\epsilon'-2-\epsilon_n}(S^l_{\epsilon}(T))=0,\textup{ for all }T\in\mathscr{M}(\Gamma_s),\epsilon,\epsilon'>0$$
    Hence we have got the desired perturbation.
\end{proof}

\appendix
\section{Appendix}
Here we collect some previous results used in this paper. The first one is an infinitesimal density drop theorem.
\begin{prop}[{\cite[Proposition 3.3]{CMSa}}]\label{density drop}
    Suppose that $C\in\mathbb{I}_n^1(\R^{n+1})$ is a non-flat minimizing cone. If $T$ is a minimizing boundary in $\R^{n+1}$ that does not cross $C$ smoothly, then $$\Theta_T(x)\leqslant\Theta_C(0),\textup{ for all }x\in\spt T.$$The equality holds if and only if $T=\pm C$ and $x\in\textup{spine }C$.
\end{prop}
The second one describes how to compare the distance of singular sets with the distance of boundaries of currents. Given a smooth, closed, oriented, $(n-1)$-dimensional $\Gamma\subset\R^{n+1}$, consider the set of possible compact minimizers with boundary $\Gamma$:$$\mathscr{M}(\Gamma)=\lbrace \textup{minimizing integral }n\textup{-currents }T\textup{ in }\R^{n+1}:\p T=[\![\Gamma]\!]\textup{ and }\spt T\textup{ compact}\rbrace.$$
\begin{thm}[{\cite[Theorem 1.7]{CMSb}}]\label{holder continuity}
    Let $\Gamma$ be a smooth closed, oriented, $(n-1)$-dimensional submanifold in $\R^{n+1}$ and $(\Gamma_s)_{s\in[-\delta,\delta]}$ be a smooth deformation of $\Gamma_0=\Gamma$. Consider the family $$\mathscr{F}=\bigcup_{s\in[-\delta,\delta]}\mathscr{M}(\Gamma_s)$$ and assume the following:
    \begin{enumerate}[(a)]
        \item All elements of $\mathscr{F}$ with distinct boundaries have pairwise disjoint supports;
        \item All elements of $\mathscr{F}$ have multiplicity one up to their boundary;
        \item There exists a hypersurface $\Sigma$ with nonempty boundary and $h:\mathscr{F}\rightarrow C^{\infty}(\Sigma)$ so that for all $s\in[-\delta,\delta], T_s\in\mathscr{M}(\Gamma_s)$, we have$$\textup{graph}_{\Sigma}h(T_s)\subset\spt T_s,\p(\textup{graph}_{\Sigma}h(T_s))=\Gamma_s.$$
        \item There exists $\alpha>0$ such that for all $s_j\in[-\delta,\delta],T_{s_j}\in\mathscr{M}(\Gamma_{s_j}),j=1,2$,$$s_1<s_2\implies h(T_{s_2})-h(T_{s_1})\geqslant\alpha(s_2-s_1)\textup{ on }\Gamma.$$
    \end{enumerate}
    Then the timestamp function $$\mathfrak{t}:\spt\mathscr{F}\rightarrow[-\delta,\delta],$$ $$\mathfrak{t}(x)=s\textup{ for all }x\in\spt T_s,T_s\in\mathscr{M}(\Gamma_s),s\in[-\delta,\delta]$$ is $\alpha$-Holder on $\sing\mathscr{F}$ for every $\alpha\in(0,\kappa_n+1)$.
\end{thm}



\begin{thebibliography}{99}
\bibitem[Alm83]{Alm}
Frederick J. Almgren Jr.
\newblock Q valued functions minimizing Dirichlet's integral and the regularity of area minimizing rectifiable currents up to codimension two.
\newblock \emph{Bulletin of the American Mathematical Society (New Serie)}.
\newblock 8(2), 1983: 327-328.

\bibitem[BDGG69]{BDGG}
E. Bombieri, E. De Giorgi and E. Giusti.
\newblock Minimal cones and the Bernstein problem.
\newblock \emph{Inventiones mathematicae}.
\newblock 7, 1969: 243–268.

\bibitem[CN13]{CN}
Jeff Cheeger and Aaron Naber.
\newblock Quantitative stratification and the regularity of harmonic maps and minimal currents.
\newblock \emph{Communications on Pure and Applied Mathematics},
\newblock 66(6),2013: 965-990.

\bibitem[CLS22]{CLS}
Otis Chodosh, Yevgeny Liokumovich and Luca Spolaor.
\newblock Singular behavior and generic regularity of min-max minimal hypersurfaces.
\newblock \emph{Ars Inveniendi Analytica}.
\newblock Paper No. 2, 2022, 27 pp.

\bibitem[CMS23a]{CMSa}
Otis Chodosh, Christos Mantoulidis and Felix Schulze.
\newblock Generic regularity for minimizing hypersurfaces in dimensions 9 and 10.
\newblock Preprintm, 2023.

\bibitem[CMS23b]{CMSb}
Otis Chodosh, Christos Mantoulidis and Felix Schulze.
\newblock Improved generic regularity of codimension-1 minimizing integral currents.
\newblock Preprint, 2023.

\bibitem[Fed69]{Feda}
Herbert Federer.
\newblock \emph{Geometric measure theory}.
\newblock Die Grundlehren der mathematischen Wissenschaften, Band 153.
\newblock Springer-Verlag New York, Inc., New York, 1969.

\bibitem[Fed70]{Fedb}
Herbert Federer.
\newblock The singular sets of area minimizing rectifiable currents with codimension one and of area minimizing flat chains modulo two with arbitrary codimension.
\newblock \emph{Bulletin of the American Mathematical Society}.
\newblock 76, 1970: 767–771.

\bibitem[FF60]{FF}
Herbert Federer and Wendell H. Fleming.
\newblock Normal and integral currents.
\newblock \emph{Annals of Mathematics (2)}.
\newblock 72, 1960 :458–520.

\bibitem[FROS20]{FROS}
Alessio Figalli, Xavier Ros-Oton, Joaquim Serra.
\newblock Generic regularity of free boundaries for the obstacle problem.
\newblock \emph{Publications mathématiques de l'IHÉS}.
\newblock 132(1), 2020: 181-292.

\bibitem[HS79]{HSa}
Robert Hardt and Leon Simon.
\newblock Boundary regularity and embedded solutions for the oriented Plateau problem.
\newblock \emph{Annals of Mathematics (2)}.
\newblock 110(3), 1979: 439-486.

\bibitem[HS85]{HSb}
Robert Hardt and Leon Simon.
\newblock Area minimizing hypersurfaces with isolated singularities.
\newblock \emph{Journal für die reine und angewandte Mathematik}.
\newblock 362, 1985: 102–129.

\bibitem[LW20]{LWa}
Yangyang Li and Zhihan Wang.
\newblock Generic Regularity of Minimal Hypersurfaces in Dimension 8.
\newblock Preprint, 2020.

\bibitem[LW22]{LWb}
Yangyang Li and Zhihan Wang.
\newblock Minimal hypersurfaces for generic metric in dimension 8.
\newblock Preprint, 2022.

\bibitem[NV20]{NV}
Aaron Naber and Daniele Valtorta.
\newblock The singular structure and regularity of stationary varifolds.
\newblock \emph{Journal of the European Mathematical Society}.
\newblock 22(10), 2020: 3305-3382.

\bibitem[Sim68]{JSim}
James Simons.
\newblock Minimal varieties in riemannian manifolds.
\newblock \emph{Annals of Mathematics (2)}.
\newblock 88, 1968: 62–105.

\bibitem[Sim83]{Sim}
Leon Simon.
\newblock \emph{Lectures on geometric measure theory, volume 3 of Proceedings of the Centre for Mathematical Analysis, Australian National University}.
\newblock  Australian National University, Centre for Mathematical Analysis, Canberra, 1983.

\bibitem[Sim08]{Simb}
Leon Simon.
\newblock A general asymptotic decay lemma for elliptic problems.
\newblock Preprint. 2008.
\bibitem[Sma93]{Sma}
Nathan Smale,
\newblock Generic regularity of homologically area minimizing hypersurfaces in eight dimensional manifolds.
\newblock \emph{Communications in Analysis and Geometry}.
\newblock 1(2), 1993: 217-228.

\bibitem[Wan22]{Wan}
Zhihan Wang.
\newblock Mean convex smoothing of mean convex cones.
\newblock Preprint, 2022.

\bibitem[Whi97]{Whi}
Brain White.
\newblock Stratification of minimal surfaces, mean curvature flows, and harmonic maps.
\newblock \emph{Journal für die reine und angewandte Mathematik}.
\newblock 1997(448), 1997: 1-36.

\end{thebibliography}
\end{document}